\documentclass{amsart}
\usepackage{amsfonts}
\usepackage{amscd}
\usepackage{amssymb}

\numberwithin{equation}{section}

\theoremstyle{plain}

\newtheorem{theorem}{Theorem}[section]

\newtheorem{lemma}[theorem]{Lemma}

\theoremstyle{definition}

\newtheorem{remark}[theorem]{Remark}

\newcommand{\R}{{\mathbb R}}
\newcommand{\Z}{{\mathbb Z}}
\newcommand{\C}{{\mathbb C}}
\newcommand{\N}{{\mathbb N}}

\newcommand{\F}{{\mathcal F}}
\newcommand{\supp}{{\operatorname {supp}\,}}
\newcommand{\cl}{{\operatorname {cl}\,}}
\newcommand{\End}{{\operatorname {End}\,}}
\newcommand{\tr}{{\operatorname {tr}\,}}

\renewcommand{\S}{{\mathcal S}}

\begin{document}
\title[Local spectral radius formulas on compact Lie groups]
{{Local spectral radius formulas on compact Lie groups}}

\author{Nils Byrial Andersen}
\address{Nils Byrial Andersen\\
Alssundgymnasiet S\o nderborg\\
Grundtvigs All\'e 86\\
6400 S\o nderborg\\
Denmark}
\email{nb@ags.dk}

\author{Marcel de Jeu}
\address{Mathematical Institute,
Leiden University,
P.O. Box 9512,
2300 RA Leiden,
The Netherlands}
\email{mdejeu@math.leidenuniv.nl}
\date{23 May, 2008}

\begin{abstract}
We determine the local spectrum of a central element of the complexified universal enveloping algebra of a compact connected Lie group at a smooth function as an element of $L^p(G)$. Based on this result we establish a corresponding local spectral radius formula.
\end{abstract}
\maketitle

\section{Introduction and statement of result}\label{sec:intro}

Let $f$ be a Schwartz function on $\R^d$ and let $P(\partial)$ be a constant
coefficient differential operator with complex coefficients. If $1\leq p\leq\infty$, then
it is known that
\begin{equation}\label{eq:limSchwartzintro}
\lim _{n\to \infty}\| P(\partial )^n f\| _p ^{1/n} =\sup \left\{ |z| : z\in \{ P(i\lambda): \lambda\in\supp \F f\}^{\cl}\right\} ,
\end{equation}
in the extended positive real numbers, were $\F f$ is the Fourier transform of $f$ and $A^{\cl}$ denotes the closure of a subset $A$ of the complex plane.
This result was first established by Tuan for real coefficients, see \cite[Theorem~2]{Tu}, and later by the authors for the general case, see \cite[Theorem~2.5]{AdJ}.

In \cite{AdJ} we raised the question whether analogues of \eqref{eq:limSchwartzintro} hold for other
Lie groups, with $P(\partial)$ replaced by an element of the center of the universal enveloping algebra,
and whether such results could be interpreted as a local spectral radius formula, analogous to the case $p=1$
on $\R^d$, see \cite[Corollary~5.4]{AdJ}. In order to explain this interpretation we recall a few relevant definitions
from local spectral theory, see \cite{ErdWang}, \cite{LaurNeu} and \cite{Vasilescu}.

Let $X$ be a Banach space, and $T: {\mathcal D}_T \to X$ a closed operator with domain ${\mathcal D}_T$.
Then $z_0\in\C$ is said to be in the local resolvent set of $x\in X$,
denoted by $\rho_T(x)$, if there is an open neighborhood $U$ of $z_0$
in $\C$, and an analytic function $\phi: U\to {\mathcal D}_T$, sending $z$ to
$\phi_z$, such that
\begin{equation}\label{eq:localresolvent}
(T-z)\phi_z=x\qquad(z\in U).
\end{equation}
The local spectrum $\sigma_T(x)$ of $T$ at $x$ is the
complement of $\rho_T(x)$ in $\C$.

The operator $T$ is said to have the
single-valued extension property (SVEP) if, for every
non-empty open subset $U\subset\C$, the only analytic solution $\phi:U\to X$ of the
equation $(T-z)\phi_z=0$ $(z\in U)$ is the zero solution. This is equivalent to requiring that the analytic
local resolvent function $\phi$ in \eqref{eq:localresolvent} is determined uniquely, so that we can speak of "the" analytic
local resolvent function on $\rho_T(x)$.

If ${\mathcal D}_T =X$ and $T$ has SVEP, then, by \cite[Proposition~3.3.13]{LaurNeu},
the local spectral radius formula
\begin{equation}\label{eq:localformulawithlimsup}
\limsup_{n\to\infty}\Vert T^n x\Vert^{1/n}=\max\left\{|z| : z\in\sigma_T(x)\right\}
\end{equation}
holds for all $x\in X$. If ${\mathcal D}_T =X$, but $T$ does not necessarily have SVEP, then by
\cite[Proposition~3.3.14]{LaurNeu} the set of $x\in X$ for
which \eqref{eq:localformulawithlimsup} holds is still always of the
second category in $X$.
If ${\mathcal D}_T =X$ and $T$ has Bishop's property ($\beta$)
(see \cite[Definition~1.2.5]{LaurNeu}; it is immediate that
property ($\beta$) implies SVEP), then, by
\cite[Proposition~3.3.17]{LaurNeu},
\begin{equation}\label{eq:localformulawithlim}
\lim_{n\to\infty}\Vert T^n x\Vert^{1/n}=\max\left\{|z| : z\in\sigma_T(x)\right\},
\end{equation}
for all $x\in X$.

Thus there exist general results concerning the validity of local spectral radius formulas, such as \eqref{eq:localformulawithlimsup} and \eqref{eq:localformulawithlim}, for bounded operators. We are not aware of such a priori guarantees for unbounded operators, and it is one of the main results in \cite{AdJ} that, for $p=1$, the equality in \eqref{eq:limSchwartzintro} can, in fact, be interpreted as a local spectral radius formula for a closed unbounded operator.\footnote{For other values of $p$ the problem is still open, although it is conjectured in \cite{AdJ} that the interpretation then holds as well.}

To be precise, let $T_{P(\partial),1}:C_c^\infty(\R^d)\to L^1(\R^d)$
be defined canonically as $T_{P(\partial),1} f=P(\partial)f$, for $f\in C_c^\infty(\R^d)$.
It is then easily seen, cf.~\cite[Section~4.2]{Schechter}, that $T_{P(\partial),1}$ has a closed
extension $\widetilde T_{P(\partial),1}$ on $L^1(\R^d)$, with domain ${\mathcal D}_{\widetilde T_{P(\partial),1}}$ consisting of those $f\in L^1(\R^d)$ such that $P(\partial)f$ is in $L^1(\R^d)$, and defined as $\widetilde T_{P(\partial),1} f=P(\partial)f$ for $f\in {\mathcal D}_{\widetilde T_{P(\partial),1}}$.

Then \cite[Corollary~5.4]{AdJ} reads as follows:

\begin{theorem}\label{cor:interpretation}
The closed operator $\widetilde T_{P(\partial),1}$ on $L^1(\R^d)$ has SVEP.
Furthermore, if $f$ is a Schwartz function on $\R^d$, then
\[
\sigma_{\widetilde T_{P(\partial),1}}(f)={\left\{P(i\lambda) : \lambda\in\supp \F f\right\}}^\cl.
\]
Combined with \eqref{eq:limSchwartzintro} this implies that the local spectral radius formula
\begin{equation}\label{eq:equivalentform}
\lim_{n\to \infty}\| {\widetilde T}_{P(\partial),1}^n f\| _1 ^{1/n} =\sup\left\{|z| : z\in\sigma_{\widetilde T_{P(\partial),1}}(f)\right\}
\end{equation}
holds in the extended positive real numbers.
\end{theorem}

This paper is concerned with the analogue of Theorem~\ref{cor:interpretation}, for $1\le p\le\infty$,
on a connected compact Lie group $G$, with Lie algebra $\mathfrak g$. We will replace $P(\partial)$ with an element
$D$ in the center of the complexified universal enveloping algebra $U(\mathfrak g)_\C$, viewed as the algebra of left-invariant differential operators on $G$. In order to state the results, we need some preliminaries
which will also be used in the proofs in the next section.

We let ${}^\dag:U(\mathfrak g)_\C\to U(\mathfrak g)_\C$ be the complex linear anti-homomorphism of
$U(\mathfrak g)_\C$ such that $X^\dag = -X$, for $X\in\mathfrak g$. If $S$ is a distribution on $G$, and
$D\in U(\mathfrak g)_\C$, then $DS$ is the distribution defined by
\[
\langle DS, \psi \rangle = \langle S, D^\dag\psi \rangle  \qquad (\psi \in C^\infty (G)).
\]
Since $G$ is unimodular, this is compatible with the action of $G$ on smooth functions.

For $1\le p\le\infty$, and $D\in U(\mathfrak g)_\C$, we define the operator $T_{D,p}: C^\infty(G)\to L^p(G)$
canonically by $T_{D,p} f= D f$, for $f\in C^\infty(G)$. Then, as in \cite[Section~4.2]{Schechter},
$T_{D,p}$ has a closed extension $\widetilde T_{D,p}$ on $L^p(G)$, with domain
${\mathcal D}_{\widetilde T_{D,p}}$ equal to those $f\in L^p(G)$
such that the distribution $Df$ is in $L^p(G)$, and defined as $\widetilde T_{D,p} f=Df$,
for $f\in {\mathcal D}_{\widetilde T_{D,p}}$.

Choose and fix representatives $(\pi,H_\pi)$ for the unitary dual $\widehat{G}$ of $G$.
If $\pi\in \widehat{G}$ (we will allow ourselves such abuse of notation), we let $\bar{\pi}$ denote its contragredient representation, and
$\chi _\pi: Z(U(\mathfrak g)_\C) \to \C$ its infinitesimal character, defined on the center $Z(U(\mathfrak g)_\C)$ of $U(\mathfrak g)_\C$.

Let $dg$ be the normalized Haar measure on $G$. If $f\in L^1 (G)$, and $\pi\in \widehat{G}$, define the Fourier transform $\F f (\pi)$ of $f$ at $\pi$ as
\begin{equation}\label{eq:fouriertransformdefinition}
\F f (\pi)= \int _G f(g) \pi (g) dg\in \End _\C (H_{\pi}).
\end{equation}
Note that $L ^p (G) \subset L ^1 (G)$, for $1\le p\le\infty$, so that the Fourier transform is defined on
$L ^p (G)$, for all $1\le p\le\infty$.

Then we have the following result:
\begin{theorem}\label{thm:liegroupspectralradiusformula}
Let $G$ be a connected compact Lie group and $D\in Z(U(\mathfrak g)_\C)$.
Let $1\le p\le\infty$.
Then the closed extension $\widetilde T_{D,p}$  of $T_{D,p}$ has SVEP.
If $f\in C^\infty (G)$, then the local spectrum of $\widetilde T_{D,p}$ at $f \in {\mathcal D}_{\widetilde T_{D,p}}$ is
given by
\begin{equation}\label{eq:liegrouplocalspectrum}
\sigma_{\widetilde T_{D,p}}(f)={\left\{\chi_{\bar{\pi}}(D) : \pi\in\supp \F f\right\}}^\cl,
\end{equation}
and the local spectral radius formula
\begin{equation}\label{eq:liegroupspectralradiusformula}
\lim_{n\to \infty}\|\widetilde T_{D,p}^n f\| _p ^{1/n} =\sup\left\{|z| : z\in\sigma_{\widetilde T_{D,p}}(f)\right\}
\end{equation}
holds in the extended positive real numbers.
\end{theorem}

\begin{remark}
Obviously, Theorem~\ref{thm:liegroupspectralradiusformula} is an analogue of Theorem~\ref{cor:interpretation}.
It would be premature to state a conjecture, but in view of these two results, the material presented in \cite{AdJ} and the proofs below, it is tempting to consider the possibility that Theorem~\ref{thm:liegroupspectralradiusformula} and Theorem~\ref{cor:interpretation} have a common generalization for Schwartz functions on connected reductive groups -- or perhaps even symmetric spaces -- including appropriate analogues of \eqref{eq:liegrouplocalspectrum} and \eqref{eq:liegroupspectralradiusformula}.
\end{remark}

\section{Proofs}\label{sec:proofs}

We now turn to the proof of Theorem~\ref{thm:liegroupspectralradiusformula}, which will occupy the remainder
of the paper. It is based on results in \cite{Sug} on the Fourier transform of smooth functions
on a connected compact Lie group, which we will now recall.

Let $G$ be a connected compact Lie group, with Lie algebra $\mathfrak g$. Choose and fix a maximal torus $T$
with Lie algebra $\mathfrak t$. Then $\mathfrak g= \mathfrak z \oplus [\mathfrak g, \mathfrak g]$, where $\mathfrak z$
is the center of $\mathfrak g$ and where $[\mathfrak g, \mathfrak g]$ is either zero or semisimple.
In the latter case, ${(\mathfrak t\cap[\mathfrak g, \mathfrak g])}_\C$ is a Cartan subalgebra of the semisimple
complex Lie algebra ${[\mathfrak g, \mathfrak g]}_\C$ and we let $\Delta$ be the roots of ${[\mathfrak g, \mathfrak g]}_\C$
relative to ${(\mathfrak t\cap[\mathfrak g, \mathfrak g])}_\C$.
Fix a choice of positive roots, and hence a set of dominant weights on ${(\mathfrak t\cap[\mathfrak g, \mathfrak g])}_\C$.

Let $\Gamma _G = \{X\in\mathfrak t : \exp X =1\}$, so that $T \cong \mathfrak t/\Gamma _G$. Then, according to
\cite[Theorem~4.6.12]{Wal}, $\widehat{G}$ is in bijective correspondence with the set $\Lambda_{\widehat{G}}$ of complex linear forms
$\lambda$ on $\mathfrak t_\C$ such that
\begin{enumerate}
\item[(1)] $\lambda (\Gamma _G) \subset 2 \pi i \Z$.
\item[(2)] $\lambda _{|{(\mathfrak t\cap[\mathfrak g, \mathfrak g])}_\C}$ is dominant integral.
\end{enumerate}
The correspondence is via highest weight modules for ${{(\mathfrak t\cap[\mathfrak g, \mathfrak g])}_\C}$, but its precise form is not relevant for the present paper. Note that if $[\mathfrak g, \mathfrak g]= 0$, i.e., if $G=T$, then the above result is still valid if one takes condition (2) to be vacuously fulfilled. As a notation in the sequel, we will let  $\lambda \in\Lambda_{\widehat{G}}$ correspond to $(\pi _\lambda, H_\lambda) \in \widehat{G}$.

The space $C^\infty (G)$ is a Fr\'echet space when equipped with the seminorms
$p_D(f) = \| Df \|_\infty,\,D\in U(\mathfrak g)_\C$, for $f\in C^\infty (G)$.
As is stated below, its counterpart on the Fourier series side is the space $\S (\widehat{G})$ of rapidly decreasing operator valued functions
on $\widehat{G}$, which we now define. Fix a norm on the dual of $\mathfrak t _\C$.  Then $\S (\widehat{G})$ is the space of functions $\phi:\Lambda_{\widehat{G}}\to \bigcup_{\pi \in {\widehat{G}}}\End _\C (H_\pi)$, such that
\begin{enumerate}
\item[(a)] $\phi (\lambda) \in \End _\C (H_{\pi_\lambda})$ for all $\lambda \in \Lambda_{\widehat{G}}$, and
\item[(b)] $\sup _{\lambda \in \Lambda_{\widehat{G}}} |\lambda |^s \| \phi(\lambda)\| < \infty$, for all $s\in \N \cup \{0\}$.
\end{enumerate}
Here, and in the sequel, the norm of an element of $\End _\C (H_{\pi})$ will always be its Hilbert--Schmidt norm. The space $\S (\widehat{G})$ becomes a Fr\'echet space when equipped with the seminorms $q_s(\phi) =\sup _{\lambda \in \Lambda_{\widehat{G}}} |\lambda |^s \| \phi(\lambda)\|$, for $\phi\in\S(\widehat{G})$ and $s\in \N \cup \{0\}$.

Let $f\in L^1 (G)$. In view of the description of $\widehat G$ above, the Fourier transform $\F f$ of $f$, as defined in \eqref{eq:fouriertransformdefinition}, can be regarded as an operator valued function on $\Lambda_{\widehat{G}}$ which satisfies (a).
With this in mind we can now give the following alternative formulation of some of the results from \cite{Sug}:
\begin{theorem}\label{thm:sugiura}
If $f\in C^\infty (G)$, then $\F f \in\S (\widehat{G})$. Moreover, the map $\F :C^\infty (G)\to \S (\widehat{G})$
is a topological isomorphism of $C^\infty (G)$ onto $\S (\widehat{G})$. The inverse map is given as
\begin{equation}\label{eq:inverseFourier}
(\F^{-1} \phi) (g) = \sum _{\pi\in\widehat{G}} \dim (\pi)\, \tr (\phi (\pi)\pi (g^{-1})) \qquad (\phi \in\S (\widehat{G}),\, g \in G),
\end{equation}
where the series converges absolutely and uniformly on $G$.
\end{theorem}
The part on absolute and uniform convergence also follows from \cite{Pe,Ta}. If $G$ is a torus, then this result specializes to a well known statement from classical Fourier analysis.

After these preparations, we can now prove Theorem~\ref{thm:liegroupspectralradiusformula} in a number of steps.
Let $D\in Z(U(\mathfrak g)_\C),\,1\le p\le \infty$, and define $\widetilde T_{D,p}$ as in the introduction. If
$f\in {\mathcal D}_{\widetilde T_{D,p}}$, then $Df \in L^p(G) \subset L^1(G)$, so that $\F(D f) (\pi)$ is defined for all
$\pi \in \widehat{G}$. Since the matrix coefficients of $\pi$ are smooth, it is easily seen that
\[
\F (\widetilde T_{D,p} f) (\pi)= \chi _\pi (D^\dag )\F  f (\pi)\qquad (f\in {\mathcal D}_{\widetilde T_{D,p}}, \pi\in\widehat{G}),
\]
which, since $\chi _\pi (D^\dag )=\chi _{\bar{\pi}} (D)$, can also be written as
\[
\F (\widetilde T_{D,p} f) (\pi)= \chi _{\bar{\pi}} (D)\F  f (\pi)\qquad (f\in {\mathcal D}_{\widetilde T_{D,p}}, \pi\in\widehat{G}).
\]
\begin{lemma}\label{lem:SVEP}
If $D\in Z(U(\mathfrak g)_\C)$ and $1\le p\le \infty$, then $\widetilde T_{D,p}$ has SVEP.
\end{lemma}
\begin{proof}
If $U\subset\C$ is open and non-empty, and
$\phi: U\to {\mathcal D}_{\widetilde T_{D,p}}$ is analytic and such that $(\widetilde T_{D,p}-z)\phi_z=0$ for
$z\in U$, then taking Fourier transforms yields $(\chi_{\bar{\pi}} (D)-z)\F \phi_z(\pi)=0$,
for all $z\in U$ and $\pi\in\widehat{G}$. If $\pi\in\widehat{G}$ is fixed, we conclude that
$\F \phi_z(\pi)=0$ for all $z\in U$ with at most one exception, which could possibly occur at $\chi_{\bar{\pi}} (D)$ if $\chi_{\bar{\pi}} (D)\in U$.
However, since $\F \phi_z(\pi)$ depends continuously on $z$, as a consequence of the continuity of the inclusion $L^p(G) \subset L^1(G)$,
such an exception does, in fact, not occur.
Hence $\F \phi_z(\pi)=0$, for all $z\in U$ and $\pi \in \widehat{G}$, so that $\phi_z=0$ for all $z\in U$
by the injectivity of the Fourier transform on $L^1(G)$.
\end{proof}

\begin{lemma}\label{lem:localspectrum}
If $f\in C^\infty (G)$, $D\in Z(U(\mathfrak g)_\C)$ and $1\le p\le \infty$, then the local spectrum
of $\widetilde T_{D,p}$ at $f$, as an element of $D _{\widetilde T_{D,p}}$, is given by
\begin{equation*}
\sigma_{\widetilde T_{D,p}}(f)=\left\{\chi_{\bar{\pi}}(D) : \pi\in\supp \F f\right\}^\cl.
\end{equation*}
\end{lemma}
\begin{proof}
We first establish that
\begin{equation}\label{eq:inclusion}
\rho_{\widetilde T_{D,p}}(f)\subset\C\setminus\left\{\chi_{\bar{\pi}}(D) : \pi\in\supp \F f\right\}^\cl.
\end{equation}
To this end, suppose that $\chi_{\bar{\pi}}(D)\in \rho_{\widetilde T_{D,p}}(f)$, for some $\pi\in\widehat{G}$.
Then there exist a neighborhood $U$ of $\chi_{\bar{\pi}}(D)$, and an analytic function $\phi:U\to {\mathcal D}_{\widetilde T_{D,p}}$
such that $(\widetilde T_{D,p} - z)\phi_z=f$, for $z\in U$.
Taking the Fourier transform at this particular $\pi$ gives
\[
(\chi_{\bar{\pi}}(D)-z)\F\phi_z(\pi)=\F f(\pi) \qquad (z\in U).
\]
Since $\chi_{\bar{\pi}}(D)$ is in $U$, we can specify $z$ at this value and conclude that $\F f(\pi)=0$ whenever
$\chi_{\bar{\pi}}(D)\in \rho_{\widetilde T_{D,p}}(f)$.
In other words, if $\chi_{\bar{\,.\,}}(D): \widehat G \to \C$ denotes the function which sends $\pi\in\widehat G$ to $\chi_{\bar{\pi}}(D)$, then
\[
\chi_{\bar{\,.\,}}(D)^{-1} \left[\rho_{\widetilde T_{D,p}}(f)\right]
\subset \widehat{G}\setminus\supp \F f,
\]
hence
\[
\chi_{\bar{\,.\,}}(D)[\supp \F f] \subset \C\setminus\rho_{\widetilde T_{D,p}}(f).
\]
Since the right hand side is closed, we conclude that
\[
\left(\chi_{\bar{\,.\,}}(D)[\supp \F f]\right)^\cl \subset \C\setminus\rho_{\widetilde T_{D,p}}(f),
\]
which is equivalent to \eqref{eq:inclusion}.

Next, we show the reverse inclusion
\begin{equation}\label{eq:reverseinclusion}
\rho_{\widetilde T_{D,p}}(f)\supset\C\setminus\left\{\chi_{\bar{\pi}}(D) : \pi\in\supp \F f\right\}^\cl,
\end{equation}
which will complete the proof.
Suppose $z_0\notin \left\{\chi_{\bar{\pi}}(D) : \pi\in\supp \F f\right\}^\cl$,
and let $\varepsilon >0$ be such that $|\chi_{\bar{\pi}}(D) - z_0| > \varepsilon$, for all $\pi\in\supp \F f$.
Let $U = \{ z\in \C : |z-z_0|< \varepsilon/2\}$, so that, for $z\in U$ and $\pi\in\supp \F f$, one has
$|\chi_{\bar{\pi}}(D) - z| > \varepsilon/2$.

Define, for each $z\in U$, the function
$\psi_z:\widehat{G}\to\bigcup_{\pi \in {\widehat{G}}}\End _\C (H_\pi)$ by
\begin{equation*}
\psi_z(\pi)=
\begin{cases}
\displaystyle{\frac{\F f(\pi)}{\chi_{\bar{\pi}}(D)-z}}&\textup{if }\pi\in\supp \F f;\\
0 &\textup{if }\pi\notin\supp \F f.
\end{cases}
\end{equation*}
Obviously $\psi_z\in \S (\widehat{G})$, since $\F f\in \S (\widehat{G})$.
It is easy to verify that the map $z\mapsto\psi_z$ is an analytic function from
$U$ to $\S(\widehat{G})$, hence, as a consequence of Theorem~\ref{thm:sugiura},
the map $z\mapsto\F^{-1}\psi_z$ is an analytic function from
$U$ to $C^\infty(G)$. Composing it with the continuous inclusion of $C^\infty(G)$
in $L^p(G)$, we obtain an analytic map $\phi : U\to {\mathcal D}_{\widetilde T_{D,p}}$ defined as
$\phi_z=\F^{-1}\psi_z$, for $z\in U$.
Since $\F [(\widetilde T_{D,p} -z)\phi_z] =\F f$ by construction, we conclude that
$(\widetilde T_{D,p} -z)\phi_z = f$, for $z\in U$. Hence $z_0\in \rho_{\widetilde T_{D,p}}(f)$ as requested.
\end{proof}

The proof of Theorem~\ref{thm:liegroupspectralradiusformula} is now completed by the following result:
\begin{lemma}\label{lem:lastlemma}
If $D\in Z(U(\mathfrak g)_\C)$, $1\le p\le \infty$ and $f\in C^\infty (G)$, then in the extended positive real numbers
\begin{equation*}
\lim_{n\to \infty}\|\widetilde T_{D,p}^n f\| _p ^{1/n} =\sup\left\{|\chi_{\bar{\pi}}(D)| : \pi\in\supp \F f\right\}.
\end{equation*}
\end{lemma}
\begin{proof}
Suppose $\pi\in\supp \F f$. Then
\begin{align*}
\left |\chi_{\bar{\pi}}(D)\right |^n\|\F f(\pi)\|&=\|\F (\widetilde T_{D,p}^n f)(\pi)\|\\
&\le \int _{G} |\widetilde T_{D,p}^n f (g)| \|\pi(g)\| dg\\
&= \dim (\pi)^{1/2}\,\|\widetilde T_{D,p}^n f (g) \|_1\\
&\le  \dim (\pi)^{1/2}\,\|\widetilde T_{D,p}^n f (g) \|_p.
\end{align*}
Since $\|\F f(\pi)\|\ne 0$, we conclude that
\begin{equation*}
|\chi_{\bar{\pi}}(D) | \le \liminf _{n \to \infty}\|\widetilde T_{D,p}^n f (g) \|_p^{1/n},
\end{equation*}
hence
\begin{equation*}
\sup\{|\chi_{\bar{\pi}}(D)| : \pi\in\supp \F f\} \le \liminf _{n \to \infty}\|\widetilde T_{D,p}^n f  \|_p.
\end{equation*}

We will now proceed to show that
\begin{equation}\label{eq:ineq}
\limsup _{n \to \infty}\|\widetilde T_{D,p}^n f  \|_p \le \sup\left\{|\chi_{\bar{\pi}}(D)| : \pi\in\supp \F f\right\},
\end{equation}
which will conclude the proof of the lemma. We may assume that the right hand side is finite. Let $g\in G$, then
\begin{align*}
\widetilde T_{D,p}^n f (g) &= (\F^{-1}\F \widetilde T_{D,p}^n f ) (g)\\
&=\sum _{\pi \in \widehat{G}}\dim (\pi)\, \tr [\F \widetilde T_{D,p}^n f(\pi) \pi(g^{-1})]\\
&=\sum _{\pi \in \widehat{G}}\dim (\pi)\, \chi_{\bar{\pi}}(D)^n \, \tr [\F f(\pi) \pi(g^{-1})].
\end{align*}
Hence
\begin{equation*}
|\widetilde T_{D,p}^n f (g) |\le \sup \left\{|\chi_{\bar{\pi}}(D)^n|:\pi\in\supp \F f\right\} \cdot
\sum _{\pi \in \widehat{G}}\dim (\pi)\, |\tr [\F (\pi) \pi(g^{-1})]|.
\end{equation*}
By Theorem~\ref{thm:sugiura}, the series is bounded by a constant $M$, uniformly in $g$, so that
\begin{equation*}
\|\widetilde T_{D,p}^n f \|_p\le M\, [\sup
\left\{|\chi_{\bar{\pi}}(D)|:\pi\in\supp \F f\right\}]^n ,
\end{equation*}
and \eqref{eq:ineq} follows.
\end{proof}

\end{document}